\documentclass{article}

\usepackage[english]{babel}

\usepackage[letterpaper,top=2cm,bottom=2cm,left=3cm,right=3cm,marginparwidth=1.75cm]{geometry}

\usepackage{amsmath}
\usepackage{graphicx}
\usepackage[colorlinks=true, allcolors=blue]{hyperref}
\usepackage{float}
\usepackage[utf8]{inputenc}
\usepackage{amssymb}
\usepackage{amsthm}
\usepackage{subfig,wrapfig}
\usepackage{ulem} 

\usepackage{physics}
\def\sign{{\textrm{sign}}}
\def\dx{{\textrm{d}x}}
\def\R{{\mathbb{R}}}

\newtheorem{theorem}{Theorem}
\newtheorem{proposition}{Proposition}
\newtheorem{corollary}{Corollary}
\newtheorem{definition}{Definition}

\usepackage{color}
\usepackage{cleveref}

\newtheorem{remark}[theorem]{Remark}
\newcommand{\zhou}[1]{{\color{blue}{[Zhou: #1]}}}

\title{Angles, orthogonality, and Pythagorean theorem in Banach spaces with two related  applications}

\author{Antonio Cicone\thanks{DISIM, Universit\`a degli Studi dell'Aquila, L'Aquila, ITALY, and and Istituto Nazionale di Geofisica e Vulcanologia, Rome, ITALY. ({\tt antonio.cicone@univaq.it})}
\and
Stefano Serra-Capizzano\thanks{Insubria University - Como, DISAT - Section of Mathematics, Como, ITALY, and Uppsala University - Uppsala, Division of Scientific Computing, Dept of Information Technology, Uppsala, SWEDEN. ({\tt s.serracapizzano@uninsubria.it})}
\and
Giacomo Tento\thanks{Insubria University - Como, DISAT - Section of Mathematics, Como, ITALY. ({\tt gtento@uninsubria.it})}
\and
Haomin Zhou\thanks{School of Mathematics, Georgia Institute of Technology, Atlanta, GA, USA. ({\tt hmzhou@math.gatech.edu})}
}

\begin{document}
\maketitle

\begin{abstract}
In the current work, we propose a generalization of angles and orthogonality from $L^2$ to generic Banach spaces, starting from a $L^p$ version of the Pythagorean theorem, $p\in [1,\infty)$. The starting point is conservation of energy measured in $L^1$ norm, as it occurs when considering the intrinsic mode functions decomposition in signal processing. This conservation of energy measure in $L^1$ norm is exactly the $L^1$ Pythagorean theorem. Besides the theoretical analysis, we apply the new notions in the context of preconditioning for structured large linear systems, by obtaining new classes of preconditioners.
The present work contains numerical experiments and various remarks on the possible use of the given framework.

\end{abstract}

\section{Introduction}\label{sec:Intro}

In this paper, building on results obtained in Signal Processing and Numerical Linear Algebra, we propose a generalization of angles and orthogonality from $L^2$ to generic Banach spaces. Such a generalization allows us to extend the Pythagorean theorem and the concept of inner product and angles to any Banach space.

In Signal Processing, and in particular in the field of non-stationary signal decomposition, it has been recently proved that if we decompose a given signal $s\in\R^n$ into simple oscillatory components, called Intrinsic Mode Functions (IMFs), using the Fast Iterative Filtering (FIF) algorithm \cite{cicone2021numerical}, the $L_1$ Fourier Energy $E_{L1}$ of $s$, defined as $E_{L1}(s)=\|\widehat{s}\|_1$, where $\widehat{s}$ represents the Discrete Fourier transform of $s$, is conserved \cite{cicone2024new}.

This conservation of energy suggests that, in some sense, the IMFs are orthogonal. However, it is well known that they are not orthogonal, in general, in the classical $L_2$ sense \cite{cicone2022multivariate}.

The idea is to introduce a new concept of orthogonality and angles in $L_1$ spaces, and, more generally, Banach spaces, that can help to explain what is observed in applications. More importantly, the questions are:  is it possible to generalize the concept of angles, orthogonality and the Pythagorean theorem to generic Banach spaces? If so, what do we gain with the new concept of orthogonality and the related versions of the Pythagorean theorem? Are there applications in which we observe such a situation? For sure, the answer is positive when considering the IMFs in the domain of non-stationary signal decomposition, while related minimization processes give surprisingly practical results in the context of approximation and preconditioning of structured linear systems \cite{Ng_book,GLT-bookI}. Motivated by these two applications in very different contexts, we would like to check whether a more general picture exists and can be described.

\section{Main results}\label{sec:Results}

It is well-known that the classical concepts of angles, inner product,  and orthogonality are not defined in Banach spaces in general. There have been studies to extend the concept of angles to operators \cite{gustafson1968angle,axler1997numerical}, and to introduce a more general concept of the inner product in what is called the semi-inner product \cite{Dragomir2004, Giles1967, lumer1961semi}. However, using these definitions to generalize the concept of angles and orthogonality to Banach spaces can be challenging for various reasons.

The idea we propose in this work is to leverage dual spaces to generalize the inner product to a Banach space and to use this more general version of the inner product to define angles between elements of Banach spaces.

\subsection{Weak inner product, the induced angle and orthogonality}

We start with the definition of a weak inner product for a generic Banach space $\mathbb{X}$ and its dual $\mathbb{X}^*$. We adopt the notation that $f^* \in \mathbb{X}^*$ is a normalized dual mapping satisfying $f^*(f) = \|f\|^\alpha$ for any $f \in \mathbb{X}$, where $\alpha > 0$ can be chosen according to $\mathbb{X}$. For example, we take $\alpha = 2$ if $\mathbb{X}$ is a Hilbert space, and $\alpha = p,(1 \leq p < \infty)$ when $\mathbb{X}$ is $L^p$. 

\begin{definition}[Weak inner product]\label{def:inner_X} 
Given $f, g \in \mathbb{X}$,  we define their weak inner product in $\mathbb{X}$ as
\begin{equation}\label{eq:Inner_X}
	\langle f, g \rangle = \frac{1}{2} [\left((f+g)^*-f^*\right)(f)+\left((f+g)^*-g^*\right)(g)],
\end{equation}
where $f^*, g^*, (f+g)^* \in\mathbb{X}^*$ are the corresponding dual mappings. 
\end{definition}

We point out that this concept of weak inner product is neither an inner product, unless $\mathbb{X}$ is Hilbert, nor a semi-inner product, as those defined in \cite{Dragomir2004, Giles1967, lumer1961semi}, because $\langle f, g \rangle$ is not necessarily linear in $f$ or $g$ in general. Using this definition, we can introduce the concept of angles between elements of a Banach space $\mathbb{X}$.

\begin{definition}[Angles in a Banach space]\label{def:angle_X} 
Given $f,\ g\in \mathbb{X}$,  we define the angle in $ \mathbb{X}$ between $f$ and $g$ as
\begin{equation}\label{eq:angle_X}
	\theta_{f,\,g}=(\widehat{f,\, g})=\textrm{arccot} \langle f, g \rangle. 
\end{equation}
\end{definition}

We can now define orthogonality in the $\mathbb{X}$ sense.
\begin{definition}[Orthogonality in a Banach space]\label{def:orth_X} 
Given $f,g\in \mathbb{X}$, they are orthogonal in the $\mathbb{X}$ sense if and only if their weak inner product in $\mathbb{X}$ is zero, or equivalently, if the angle between them in the $\mathbb{X}$ space, $(\widehat{f,\, g})$,  equals $\frac{\pi}{2}$.
\end{definition}

Hence, the generalization of the Pythagorean Theorem to Banach spaces reads as follows.
\begin{theorem}[Generalized Pythagorean Theorem]\label{thm:pytha_X} 
Given $f,\ g\in \mathbb{X}$, where $\mathbb{X}$ is a Banach space, and $f^*, g^*, (f+g)^* \in\mathbb{X}^*$ with $\mathbb{X}^*$ being its dual space, $f$ and $g$ are orthogonal in the sense of $ \mathbb{X}$, if and only if
\begin{equation}\label{eq:pytha_X}
	\|f + g\|^\alpha = \|f\|^\alpha + \|g\|^\alpha.
\end{equation}
\end{theorem}
\begin{proof} According to the definition \eqref{def:inner_X}, we have 
\begin{align*}
        \|f + g\|^\alpha - \|f\|^\alpha - \|g\|^\alpha & = (f+g)^* (f+g) - f^*(f) - g^*(g) \\
        & = ((f+g)^*-f^*)(f) - ((f+g)^*-g^*)(g) \\
        & = 2 \langle f, g \rangle = 0. 
\end{align*}
\end{proof}

Unlike the inner product defined in a Hilbert space, which is bilinear, the weak inner product defined in \eqref{def:inner_X} is not bilinear, unless $\mathbb{X}$ becomes a Hilbert space. Here are some properties of the weak inner product that can be verified directly from its definition.

\begin{proposition}
Let $f, g$ be in a Banach space $\mathbb{X}$, and $\beta \in \mathbb{R}$. The following statements hold. 
\begin{enumerate}
    \item $\langle f, g \rangle =  \langle g, f \rangle$,
    \item $ \langle f, f \rangle = \|f\|^\alpha$,
    \item $\langle \beta f, \beta g \rangle = \beta^\alpha \langle f, g \rangle$,
    \item $\langle f, g \rangle = \frac{1}{2} (\|f + g\|^\alpha - \|f\|^\alpha - \|g\|^\alpha)$.
\end{enumerate}
\end{proposition}

\begin{theorem} \label{thm:birkhoff}
    Let $f, g$ be in a Banach space $\mathbb{X}$. $f$ and $g$ are orthogonal in the $\mathbb{X}$ sense implies that $f$ is Birkhoff orthogonal to $g$, i.e. $\|f + \beta g\| \geq \|f\|$ for any $\beta \in \mathbb{R}$.  
\end{theorem}
\begin{proof}
    Since $\langle f, g \rangle = 0$, we have 
    $$
    \|f+\beta g\|^\alpha  = \|f\|^\alpha + \|\beta g\|^\alpha \geq \|f\|^\alpha. 
    $$
    This gives $\|f + \beta g\| \geq \|f\|$.  
\end{proof}
We note that the converse statement of Theorem \ref{thm:birkhoff} is not true in general. 

The definitions of angles and orthogonality presented in this section can be extended to normed linear spaces. For simplicity, we focus our discussion in this paper on Banach spaces and postpone its further generalization to normed linear spaces to future work.

\subsection{Special cases}

In the following, we give a short description of three cases of interest.

\subsubsection{Hilbert space}

If the Banach space becomes a Hilbert space, for example, $\mathbb{X}=L^2$, if we take $\alpha = 2$, the weak inner product proposed in Definition \ref{def:inner_X} coincides with the classical inner product which induces the norm in the Hilbert space. This is guaranteed by the presence of $\frac{1}{2}$ in front of \eqref{eq:Inner_X}.

\subsubsection{\texorpdfstring{$L^p$}{Lp} spaces}

In the case of an $L^p$ space, for $1< p < \infty$, and its dual is $L^q$, such that $\frac{1}{p}+\frac{1}{q}=1$. The dual mapping of $f \in L^p$ is given by 
$$
f^*(x) = |f(x)|^{p-2}f(x). 
$$
This suggests that we can take $\alpha = p$, because 
$$
f^*(f) = \int f^*(x) f(x) dx = \|f\|^p.
$$

Following Definition \ref{def:orth_X}, given $f, g \in L^p$ and $f^*,\ g^*,\ (f+g)^* \in L^q$ then $f$ and $g$ are orthogonal in the $L^p$ sense if
\begin{equation}\label{eq:InnerProd}
	\langle f, g \rangle_p = \frac{1}{2}\int f(x)\left((f(x)+g(x))^*-f^*(x)\right)+g(x)\left((f(x)+g(x))^*-g^*(x)\right) \dx=0.
\end{equation}

\subsubsection{\texorpdfstring{$L^1$}{L1} space}

In the case of the $L^1$ space and its dual $L^\infty$, we can define the angle between two elements in $L^1$ by introducing a weak inner product using the sign function defined in the dual space. From Definition \ref{def:inner_X} it follows that

\begin{definition}[Weak inner product in $L^1$]\label{def:inner_L1} 
Given $f,\ g\in L^1$ and $\sign(f(x)+g(x))$, $\sign(f(x))$, $\sign(g(x))\in L^\infty$ we define their weak inner product in $L^1$ as
\begin{equation}\label{eq:Inner_L1}
	\langle f, g \rangle_1 = \frac{1}{2}\int f(x)\left(\sign(f(x)+g(x))-\sign(f(x))\right)+g(x)\left(\sign(f(x)+g(x))-\sign(g(x))\right) \dx.
\end{equation}
\end{definition}

From this definition, we can derive the concept of the angle between two elements in $L^1$.
\begin{definition}[Angles in $L^1$]\label{def:angle_L1} 
Given $f,\ g\in L^1$ and $\sign(f(x)+g(x)),\ \sign(f(x)),\ \sign(g(x))\in L^\infty$ we define the angle in $L^1$ between $f$ and $g$ as
\begin{equation}\label{eq:angle_L1}
	\theta_{f,\,g}=(\widehat{f,\, g})_1=\textrm{arccot} \langle f, g \rangle_1.
\end{equation}
\end{definition}

The orthogonality follows. 
\begin{corollary}[Orthogonality in $L^1$]\label{cor:orth_L1}
Given $f,g\in L^1$, their weak inner product in $L^1$ is zero if and only if they are orthogonal in the $L^1$ sense.
\end{corollary}

We are now ready to present the Pythagorean Theorem for the $L^1$ space.
\begin{theorem}[Pythagorean theorem in $L^1$]\label{thm:pytha_L1} 
Given $f,\ g\in \mathbb{X}$, they are orthogonal in the $L^1$ sense, if and only if
\begin{equation}\label{eq:pytha_L1}
	\|f+g\|_1=\|f\|_1+\|g\|_1.
\end{equation}
\end{theorem}

\begin{theorem} \label{thm:Schwarzp} Given $f$ and $g$ in a Banach space $\mathbb{X}$, if $\alpha = p$ where $p$ is an integer larger than $1$, then
\begin{equation}
    |\langle f, g \rangle| \leq \sum_{k=1}^{p-1} \binom{p}{k} \|f\|^k \|g\|^{p-k}. 
\end{equation}
Particularly if $\alpha = 2$, the weak inner product satisfies 
\begin{equation}
    |\langle f, g \rangle| \leq \|f\| \|g\|. 
\end{equation}
\end{theorem}
\begin{proof}
  For an integer $p$ larger than $1$, by the binomal formula, we have 
  \begin{align*}
      \langle f, g \rangle & = \frac{1}{2} (\|f+g\|^p - \|f\|^p - \|g\|^p) \\ & \leq \sum_{k=1}^{p-1} \binom{p}{k} \|f\|^k \|g\|^{p-k}.
  \end{align*}
\end{proof}

Similarly, by the triangle inequality, we have the following statement.  
\begin{theorem} \label{thm:Schwarz1} Given $f$ and $g$ in $L^1$, in which $\alpha = 1$, the weak inner product satisfies 
\begin{equation}
    \langle f, g \rangle \leq 0. 
\end{equation}
\end{theorem}

This last statement implies that the angle between any two elements in $L^1$ is not smaller than $\frac{\pi}{2}$. 

\section{Examples}

The section is divided into two parts. In the first, we give an application in signal processing, while in the second, we consider the preconditioning problem for the fast solution of structured large linear systems.

\subsection{Signal Processing}

In the field of Signal Processing of non-stationary signals, two decades ago, the so-called Iterative Filtering (IF) algorithm was proposed \cite{lin2009iterative} as a stable and convergent alternative to the famous Empirical Mode Decomposition method \cite{huang1998empirical}. These methods decompose a given signal into the so-called Intrinsic Mode Functions (IMFs), which are simple oscillatory mono-component signals. The decomposition is done iteratively. These algorithms extract the highest frequency component contained in the signal $s(t)$ using a sifting operator $\mathcal{M}$. In Iterative Filtering, the sifting operator corresponds to the convolution with a filter function $w(t)$ which is an even, positive, compactly supported function with integral equal to 1 and whose support is scaled automatically by IF based on the signal itself. So $\textrm{IMF}_1$ is obtained by iterating $s_{n+1}=s_n-\mathcal{M} (s_n)$ until a stopping criterion is fulfilled. Then $\textrm{IMF}_1$ is subtracted from the signal, and the entire procedure is repeated as it is to extract all subsequent IMFs, until no more oscillations are left in the remainder $r$ \cite{lin2009iterative}. In recent years, IF has been accelerated in the discrete setting in what is now known as the Fast Iterative Filtering (FIF) technique \cite{cicone2021numerical}. All these techniques are well known not to conserve the energy of the signal in a classical $L^2$ sense. In fact, given a signal $s\in L^2$ and its decomposition into $N$ IMFs plus a trend $r$, then in general $\|s\|_2^2 \neq \sum_{k=1}^N \|\textrm{IMF}_k\|_2^2+\|r\|_2^2$. This follows from the fact that, in general, the IMFs are not orthogonal to each other \cite{cicone2022multivariate}.

In the following, we consider a signal as a finite sequence $s=(s_j)_{0}^{n-1}$ in $\mathbb{R}^n$. This finite sequence should be viewed as a time series over the time interval $[0, 1]$, evenly sampled at the rate of $2B$. More precisely, for $j=0, 1, \dots,  n-1$, $n=2B$,
\begin{equation}
    t_j =\frac{j}{2B},\quad s_j = s(t_j).\label{eq:t_j}
\end{equation}

Thus, $(s_j)$ can be viewed as a discretization of $s$. With a bit of abuse of notation, $s$ refers to both the underlying signal and the discretization of the signal.

The discrete Fourier transform of $(s_j)$, $\hat{s}_k=\hat{s}(\xi_k)$, where
\begin{equation}
\xi_k=k, \quad k=0,1,\dots,n-1\label{eq:xi_k}
\end{equation}
and
\begin{equation}
\hat{s}(\xi_k)=\sum_{j=0}^{n-1}s_je^{-2{\pi}\hat i t_j{\xi}_k}.
\end{equation}
So the sampling rate for $\hat{s}$, the discrete Fourier transform of $s$, is $1$ over the frequency interval $[0,2B]$.

In a recent paper \cite{cicone2024new}, the authors proposed the following new definition of energy and its conservation.

\begin{definition}[$L_1$ Fourier Energy of a signal $s$]\label{def:L1_Energy}
Given a signal $s$, its $L_1$ \textit{Fourier Energy} is defined as
\begin{equation}\label{eq:L1_Energy}
E_{1}(s)=||\widehat{s}||_1.
\end{equation}
In particular, if $s=(s_j)\in\mathbb{R}^n$, then $E_1(s) =
\displaystyle{\sum_k |\widehat{s}(\xi_k)|}$ where $\xi_k$ is defined as in (\ref{eq:xi_k}).
\end{definition}

\begin{definition}[Conservation of $L_1$ Fourier Energy]\label{def:L1_Energy_cons}
Let $s\in\mathbb{R}^n$ and $s=\sum_k \phi_k$ be a decomposition of $s$. We say that the decomposition conserves the signal $L_1$ Fourier Energy if and only if
\begin{equation}\label{eq:Conservation_L1_Energy}
E_{1}(s)=\sum_k E_{1}(\phi_k)
\end{equation}
\end{definition}

They also proposed a definition of unwanted oscillations produced in a decomposition.

\begin{definition}[Unwanted oscillations] \label{def:Unwanted_oscillations}
Let $s\in\mathbb{R}^n$ and $s=\sum_k \phi_k$ be a decomposition of $s$. The decomposition $\{\phi_k\}$ contains unwanted oscillations if there exist $\xi$ such that
\begin{equation}\label{eq:Unwanted_oscillations}
    \sum_k \left|\widehat{\phi}_k(\xi)\right| > \left|\widehat{s}(\xi)\right|
\end{equation}
\end{definition}

And they proved the following.

\begin{theorem}[FIF decomposition $L_1$ Fourier energy conservation]\label{thm:Energy_conservation}
Let $s\in \mathbb{R}^n$ and $\delta >0$. Apply the FIF algorithm with $w$ a double convolution filter, and we have $s=\displaystyle{\sum_1^m \textrm{IMF}_k + r}$, where $r$ is a trend. Then this decomposition conserves the $L_1$ Fourier energy and produces no unwanted oscillations.
\end{theorem}

This theorem suggests that the IMFs produced by the FIF are, in some sense, orthogonal to each other. Thanks to the definitions we presented in the previous section, we can now easily prove the following.

\begin{theorem}[IMFs orthognality in the $L^1$ sense] \label{thm:IMFs_orthogonality}
Let $s\in \mathbb{R}^n$ and $\delta >0$. Apply the FIF algorithm with $w$ a double convolution filter, which produces the decomposition $s=\displaystyle{\sum_1^m \textrm{IMF}_k + r}$, where $r$ is a trend. Then, the IMFs are orthogonal to each other in the $L^1$ sense.
\end{theorem}

The proof is a direct consequence of Theorem \ref{thm:Energy_conservation}, and in particular, the conservation of the $L^1$ energy of the signal when decomposed via FIF, and the Pythagorean Theorem in $L^1$, Theorem \ref{thm:pytha_L1}.

It is important to observe here that, different from $L^2$, in this setting we no longer have uniqueness in the orthogonal decomposition. We can see this clearly from the following example.

\subsubsection{Numerical example}

If we consider the following non-stationary signal

$$s(x) = \cos\left(2\pi\left(90x^2+10x\right)\right)+\cos(2\pi x), \qquad x\in [0,\, 1],$$
shown in \Cref{fig:Ex_1_sig}, left panel, it can be decomposed using the FIF algorithm\footnote{FIF algorithm is available for Matlab, Python and C at \url{www.cicone.com}}. Depending on the choice of the parameters in FIF, the decomposition may vary a lot, ref. \Cref{fig:Ex_1_sig} central and right panel. However, from \Cref{thm:IMFs_orthogonality} we know that all the decompositions produced using FIF provide a set of IMFs that are orthogonal to each other in the $L^1$ sense, no matter the choice of the parameters in FIF. This is a really important feature of the IF methods, which allows for flexible decompositions. Non-stationary signals, in fact, should be decomposed in different ways, depending on the application under investigation. In particular, if the signal is related to strongly non-stationary phenomena, like, for instance, gravitational wave signals \cite{Longo2025AdaptiveSL}, then it is preferable to capture the strongly non-stationary component related to the passage of a gravitational wave in a single component, like what is shown in the right panel of \Cref{fig:Ex_1_sig}. Whereas, if the goal is the analysis of which frequencies are active when, like, as an example, in the study of geophysical signals related to the upper atmosphere \cite{spogli2025investigating}, then we want to decompose the signal into quasi-stationary frequency bands. This latter is the approach used in the decomposition of the signal $s$ shown in the central panel of \Cref{fig:Ex_1_sig}.

\begin{figure}[H]
	\includegraphics[width=0.3\textwidth]{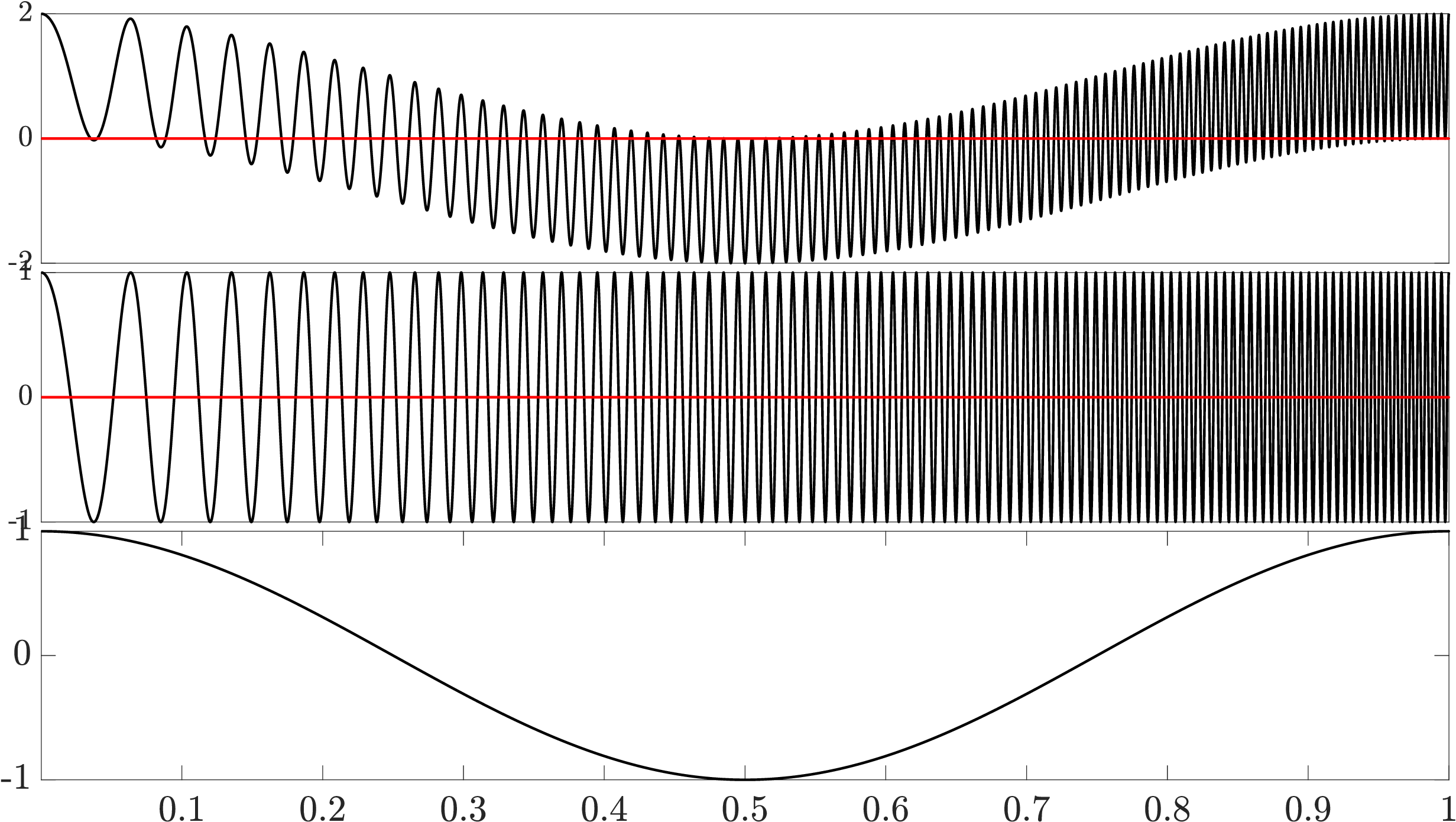}
    \includegraphics[width=0.35\textwidth]{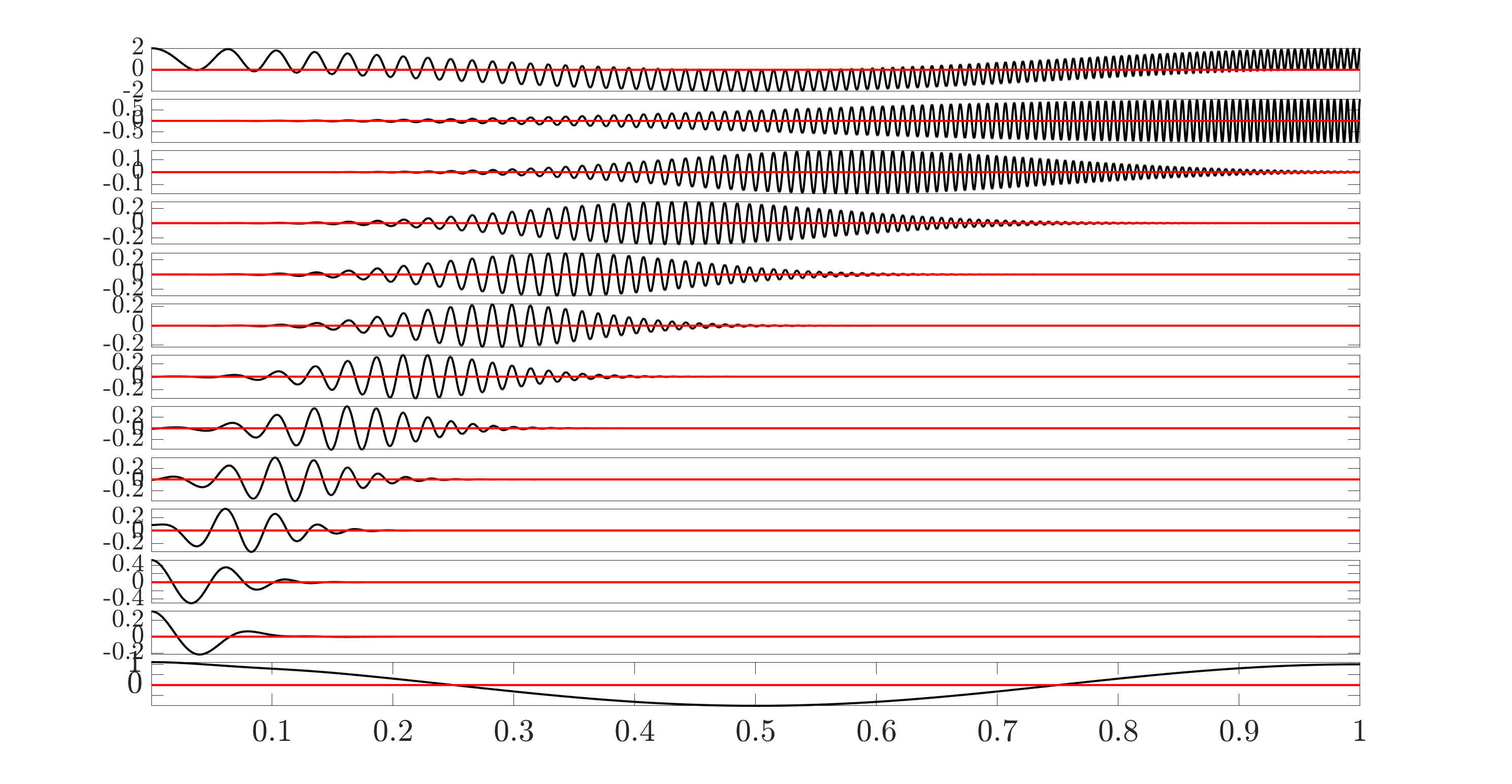}
    \includegraphics[width=0.3\textwidth]{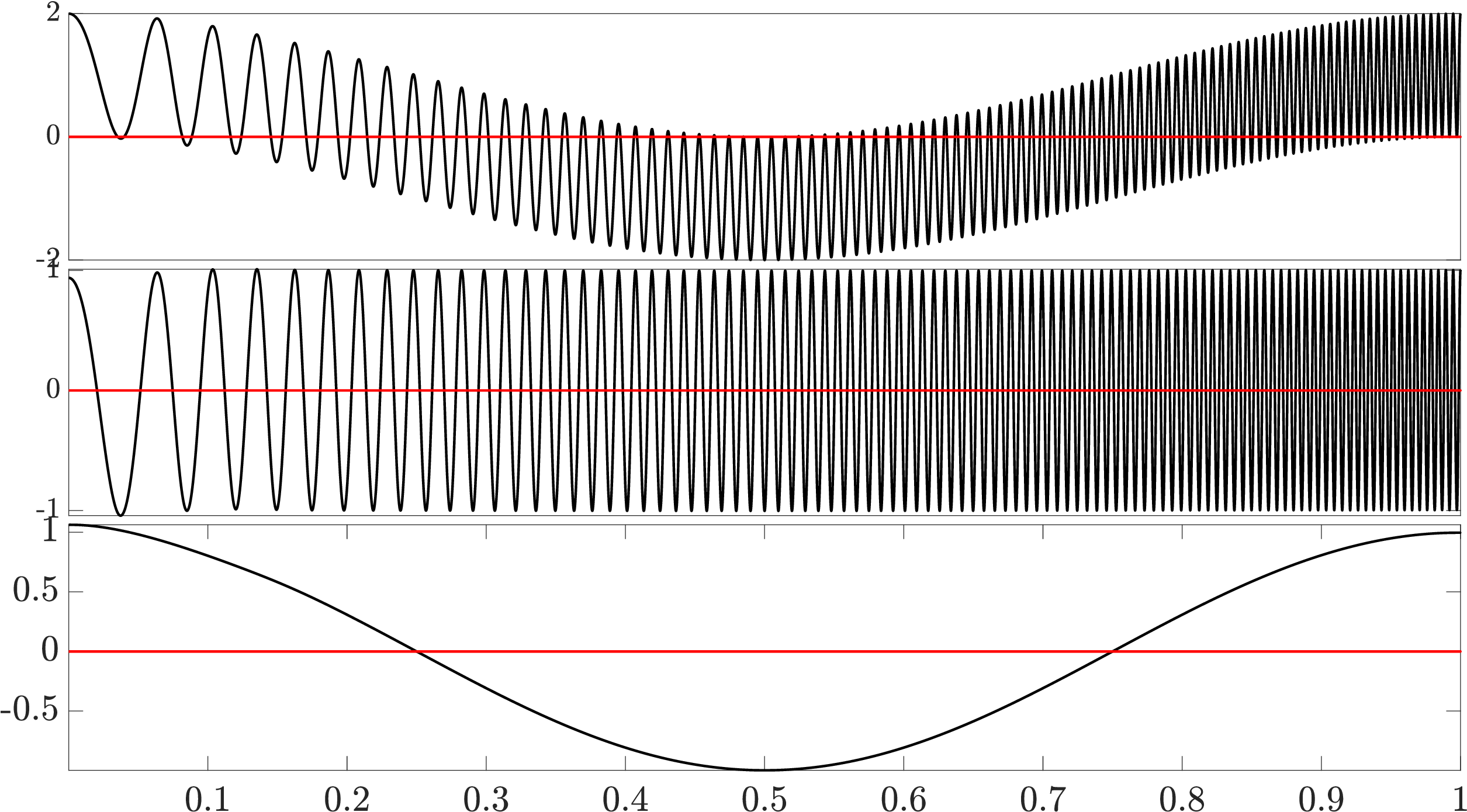}
	\caption{Left panel, the signal $s$ and its mono-frequency non-stationary components, first row and subsequent ones, respectively. Central panel, FIF decomposition using default parameter values ($\alpha = 30$, and $\Xi = 1.6$). Right panel, FIF decomposition with $\alpha = 100$, and $\Xi = 3$.}
	\label{fig:Ex_1_sig}
\end{figure}

We can observe from \Cref{fig:Ex_1_L1_energy_quasi_stat,fig:Ex_1_L1_energy_non_stat} that both decompositions conserve, up to machine precision, the $L^1$ Fourier Energy of the signal, see right panels in both figures, as expected from \Cref{thm:Energy_conservation}. This implies, by \Cref{thm:IMFs_orthogonality}, that the IMFs are also orthogonal to each other in the $L^1$ sense.

\begin{figure}[H]
	\includegraphics[width=0.32\textwidth]{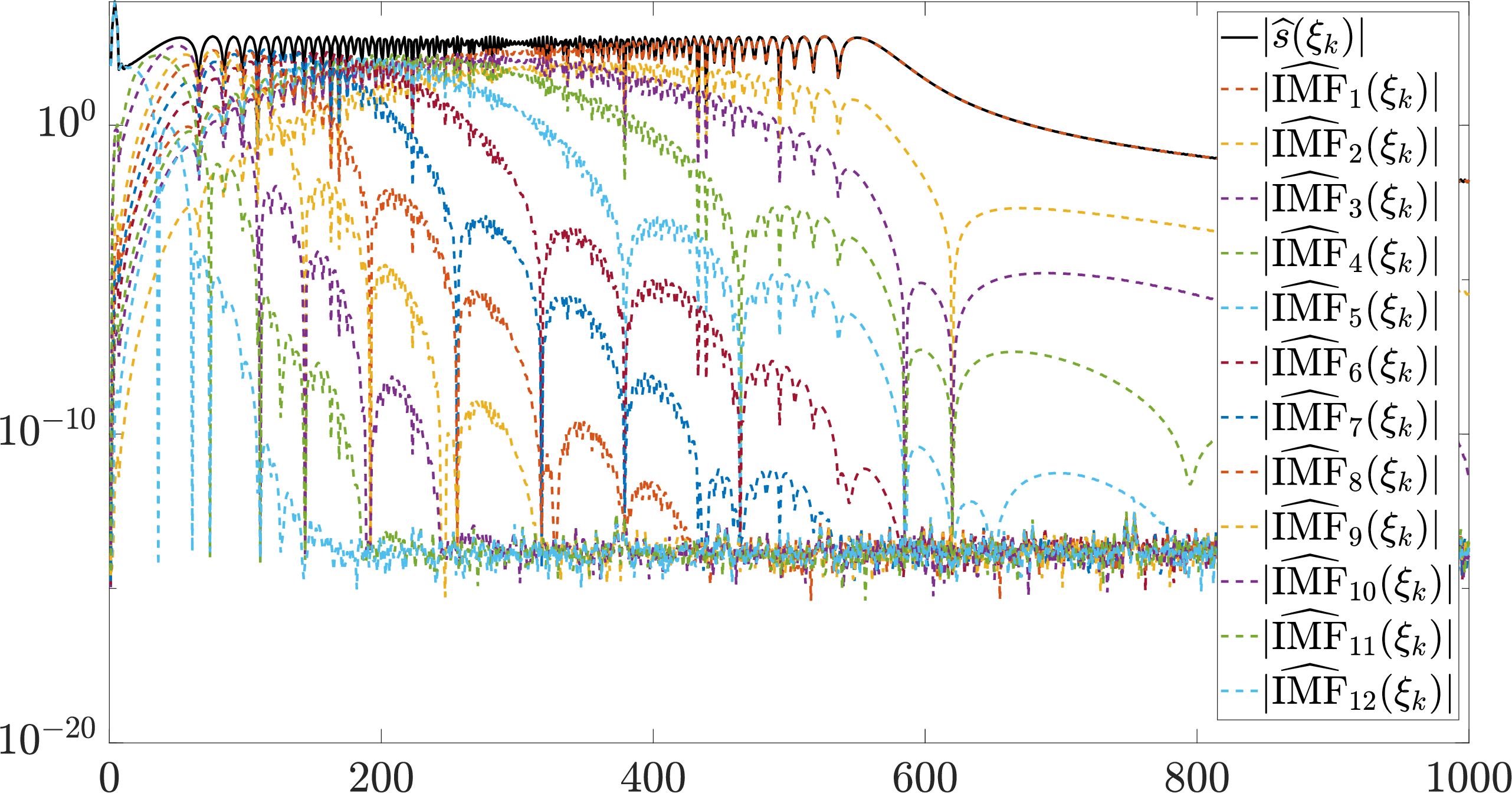}
    \includegraphics[width=0.32\textwidth]{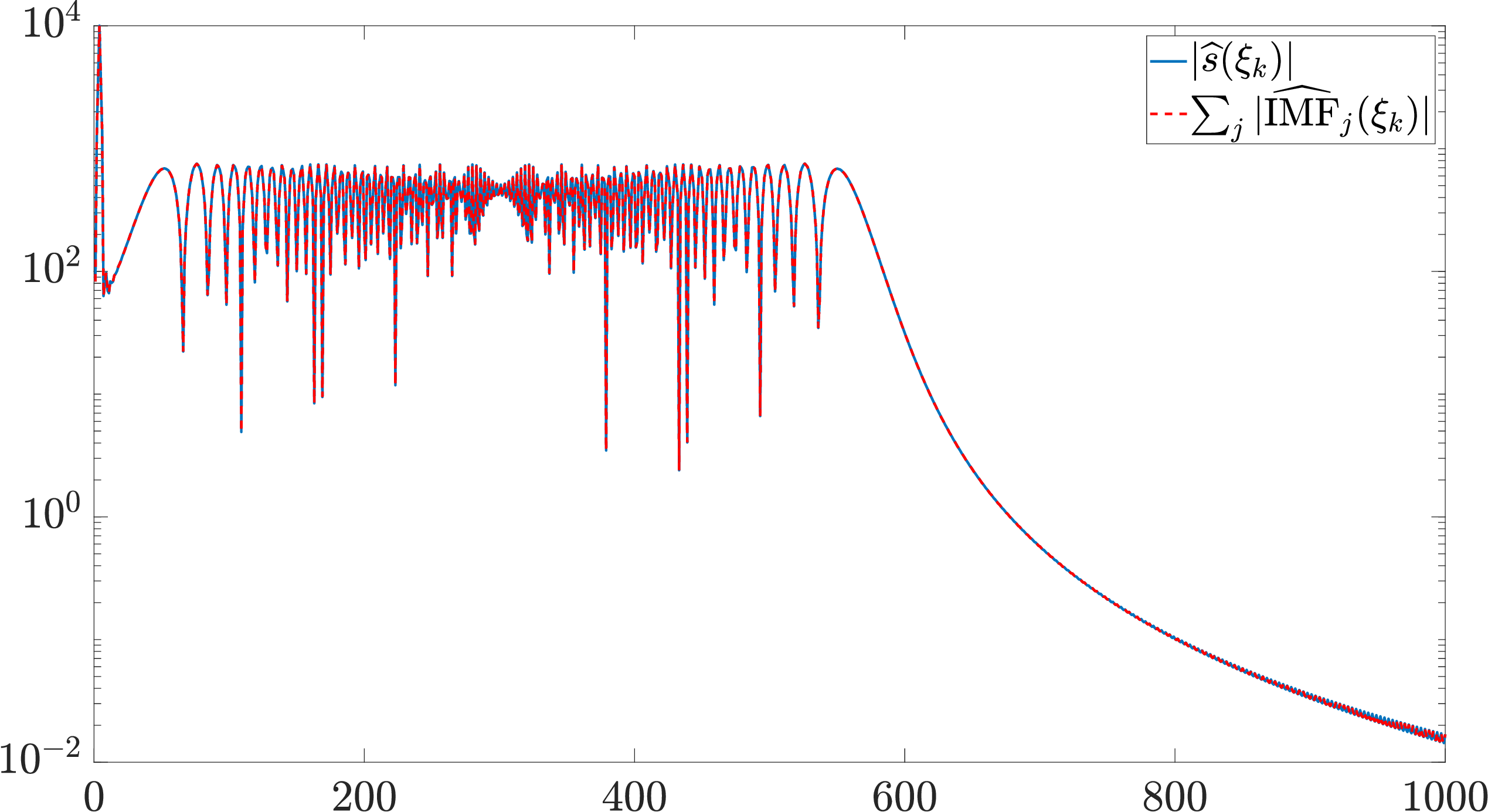}
    \includegraphics[width=0.32\textwidth]{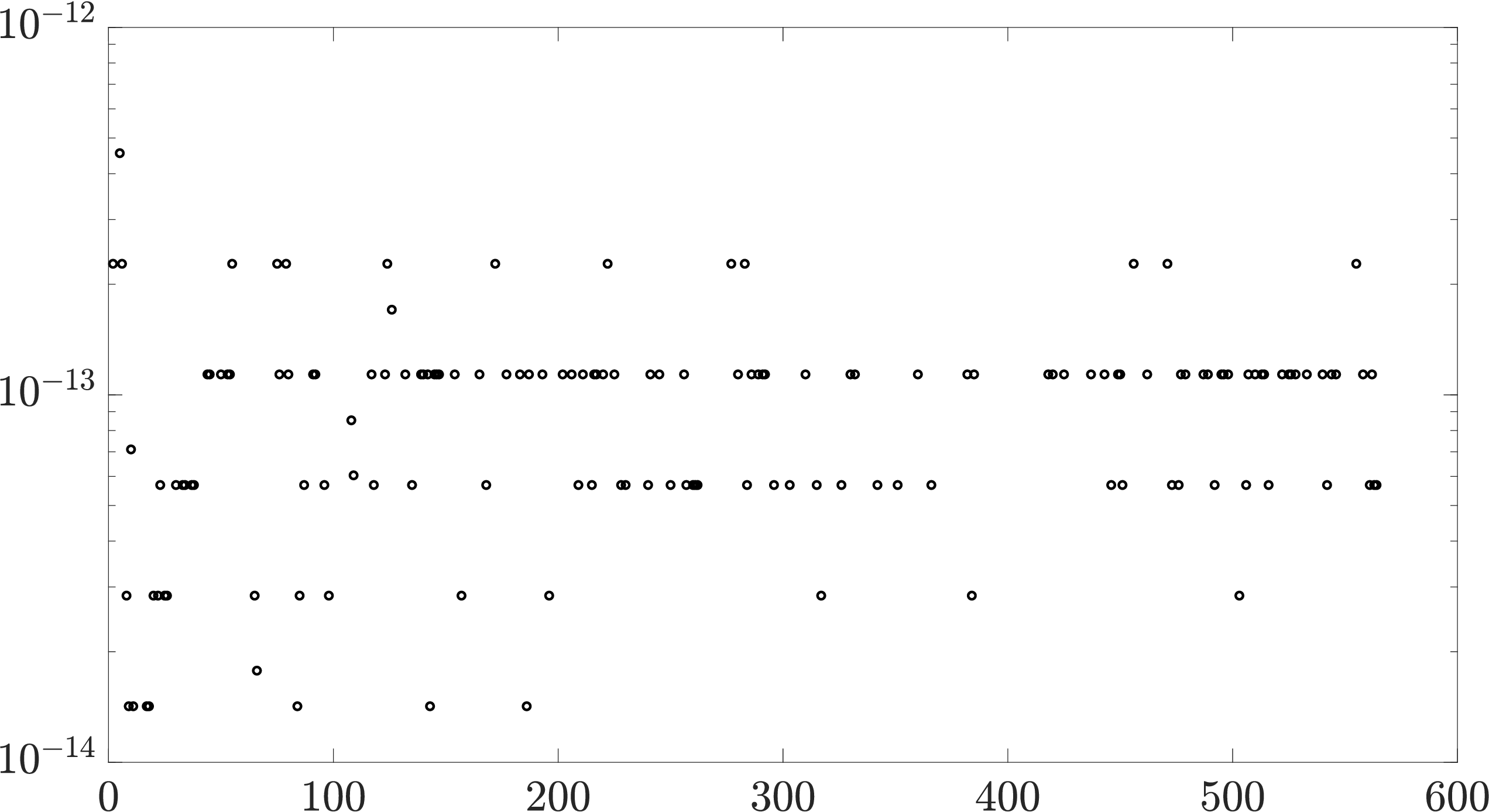}
	\caption{Left panel, the discrete Fourier transform absolute values of the signal $s$ and the IMFs produced by FIF with default parameter values $\alpha = 30$, and $\Xi = 1.6$, and displayed in the central panel of \Cref{fig:Ex_1_sig}. Central panel, the discrete Fourier transform absolute values of the signal $s$ and the summation of the discrete Fourier transform absolute values of the IMFs. Right panel, the difference between the discrete Fourier transform absolute values of the signal $s$ and the summation of the discrete Fourier transform absolute values of the IMFs. \zhou{Do the pictures in this figure mis-matched with the caption.}}
	\label{fig:Ex_1_L1_energy_quasi_stat}
\end{figure}

\begin{figure}[H]
	\includegraphics[width=0.32\textwidth]{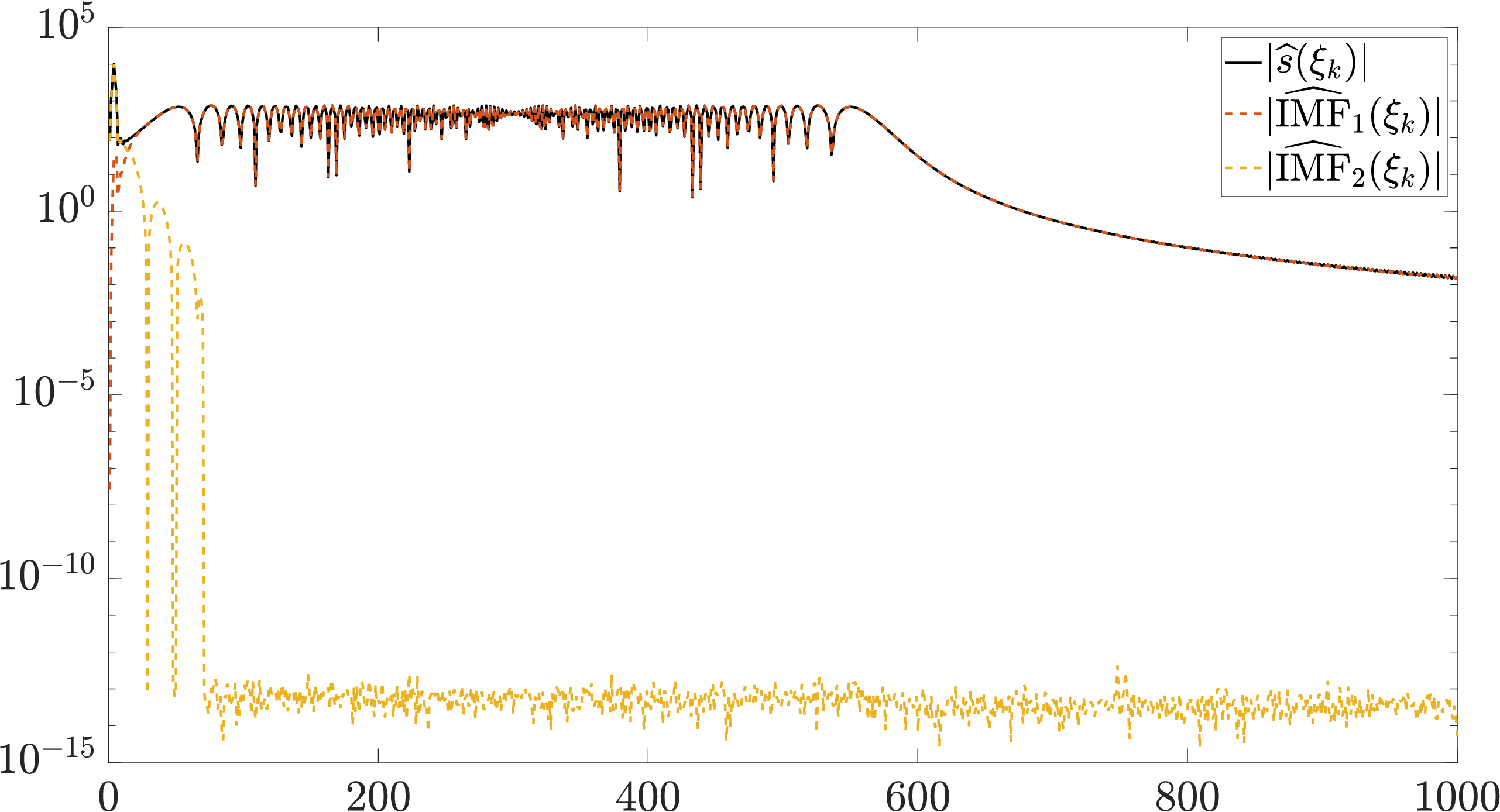}
    \includegraphics[width=0.32\textwidth]{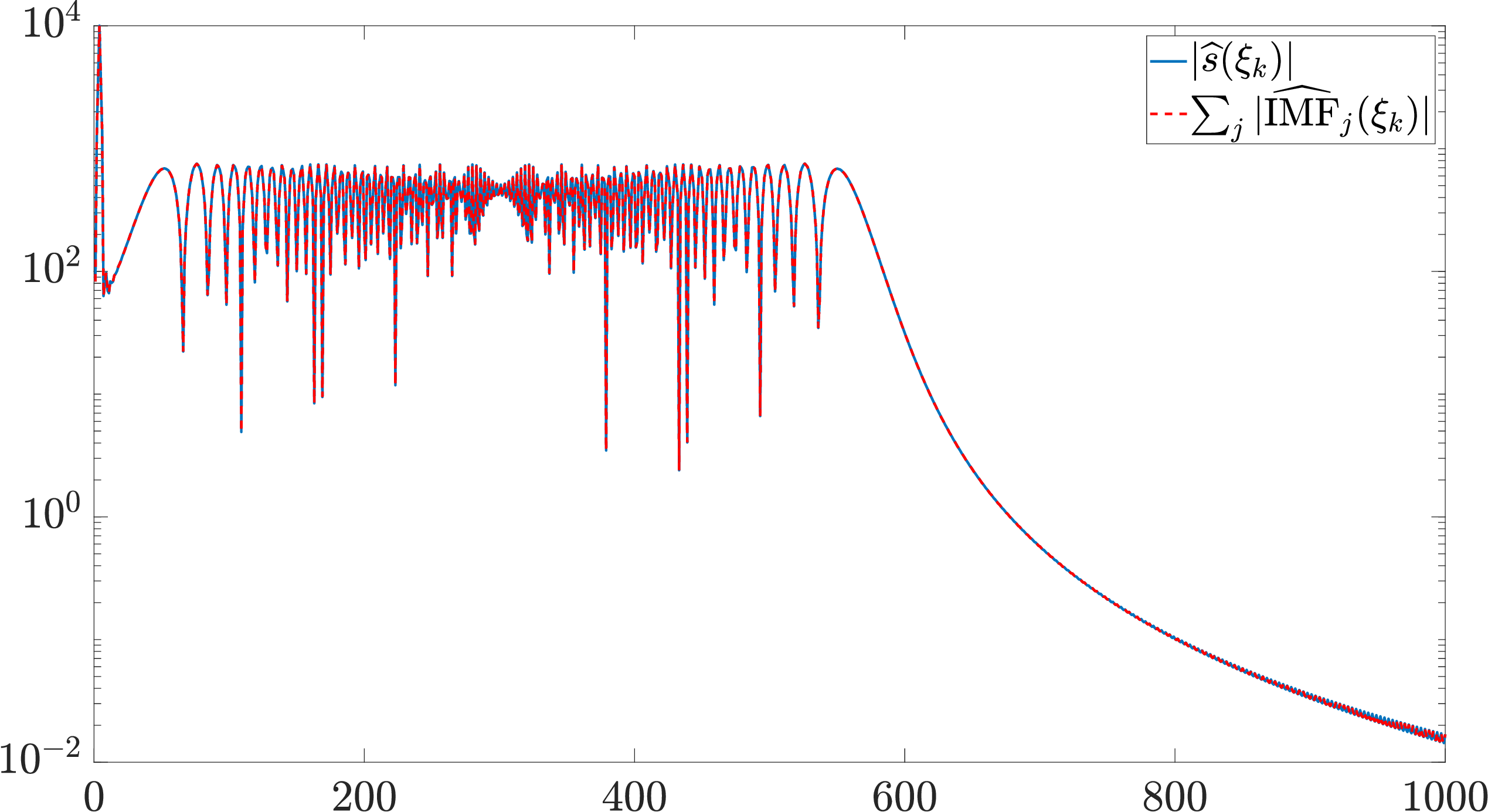}
    \includegraphics[width=0.32\textwidth]{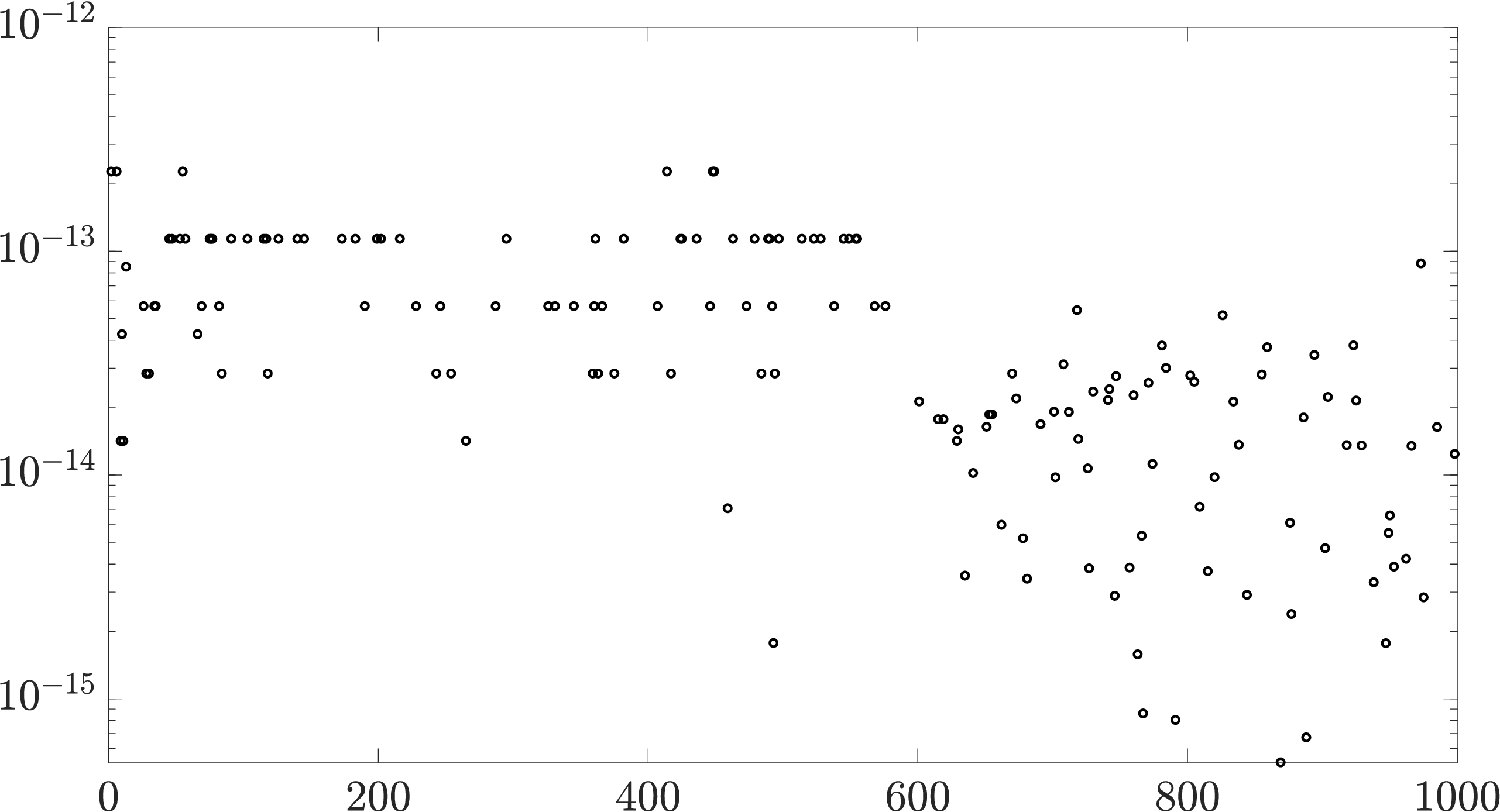}
	\caption{Left panel, $|\widehat s(\xi_k)|$ and $\left\{|\widehat{\textrm{IMF}}_j(\xi_k)|\right\}$ of the IMFs produced by FIF with parameter values $\alpha = 100$, and $\Xi = 3$, and displayed in the central panel of \Cref{fig:Ex_1_sig}. Central panel, $|\widehat s(\xi_k)|$ compared with $\sum_j |\widehat{\textrm{IMF}}_j(\xi_k)|$. Right panel, the difference between $|\widehat s(\xi_k)|$ and $\sum_j |\widehat{\textrm{IMF}}_j(\xi_k)|$.}
	\label{fig:Ex_1_L1_energy_non_stat}
\end{figure}

The importance of the fact that IF methods can produce several decompositions of the same signal and that they are all orthogonal in the $L^1$ sense is evident if we consider the well-known trigonometric identities product-to-sum and sum-to-product, also known as prosthaphaeresis formulas. From them, it is obvious that signals may be interpreted both as a single frequency with modulated amplitude or as a summation of two stationary frequencies, both with constant amplitude.

\subsection{Numerical Linear Algebra}

When considering the problem of solving large, structured linear systems in a fast way, one popular technique is the Frobenius optimal approximation, which leads to the design of preconditioners when the preconditioning matrix is constrained in some special linear space of matrices.

Among the most popular proposals, we can mention:
\begin{itemize}
\item The approximate sparse inverse preconditioning, where the linear space is given by a convenient sparsity pattern.
\item The Frobenius optimal preconditioning in fast matrix algebras, such as circulants, $\omega$-circulant, and trigonometric matrix algebras.
\end{itemize}
While the first is often used in the context of (multilevel) band structures, often encountered in the context of approximation of partial differential equations or graph problems \cite{sparse_appr1,sparse_appr2}, the second represents a solid research line in the context of structured problems of (multilevel) Toeplitz type and of their varying counterparts; see \cite{GLT-bookI,Ng_book} and references therein.

The current section is divided into two parts. In the first, we treat the minimization process, and in the second, we consider its practical use in terms of new preconditioning proposals, with a plan for future related investigations.

\subsubsection{$ l^{ p } $ circulant minimizers for Toeplitz matrices}

{Given ${n} \in \mathbb{N}$, a matrix of the form
\[
[A_{i,j}=A_{{i}-{j}}]_{{i},{j}=1}^{{n}} \in \mathbb{C}^{{n} \times {n}},
\]
with entries $A_{{k}} \in \mathbb{C}$, ${k} \in [-({n-1}), \dots, {n-1}]$, is called Toeplitz matrix or shift-invariant matrix.
A specific subclass of the Toeplitz matrices is that of circulant matrices which are characterized by a periodic shift-invariance, that is, $A$ is a $n\times n$ circulant if and only if $[A_{i,j}=A_{({i}-{j}){\rm mod}\, n}]_{{i},{j}=1}^{{n}} \in \mathbb{C}^{{n} \times {n}}$ with $A_{{k}} \in \mathbb{C}$, ${k} \in [0, \dots, {n-1}]$. Of course, a circulant matrix is also Toeplitz, while the converse is not true.

Given a complex-valued function $f : [-\pi, \pi] \to \mathbb{C}$ belonging to $L^1([-\pi, \pi])$, the ${n}$-th Toeplitz matrix associated with $f$ is defined as
\[
T_{{n}}(f) := [\hat{f}_{{i-j}}]_{{i,j=1}}^{{n}} \in \mathbb{C}^{{n} \times {n}},
\]
where
\[
\hat{f}_{{k}} = \frac{1}{2\pi} \int_{[-\pi,\pi]} f({\theta}) e^{-\hat i{k}\theta} d {\theta} \in \mathbb{C}, \quad {k} \in \mathbb{Z},
\]
are the Fourier coefficients of $f$, in which $\hat i$ denotes the imaginary unit.
$\{T_{{n}}(f)\}_{{n} \in \mathbb{N}}$ is the family of Toeplitz matrices associated with $f$, which is called the \textit{generating function}; see \cite[Chapter 6]{GLT-bookI} and references there reported.}

Circulant matrices represent an attractive way for approximating Toeplitz matrices, especially when they are extracted from a matrix sequence $\{T_{{n}}(f)\}_{{n} \in \mathbb{N}}$ associated with a generating function. In fact, circulant matrices of size $n$ form a $*$-algebra of matrices simultaneously diagonalized by the discrete Fourier transform of dimension $n$ and therefore any matrix operation has $O(n\log n)$ arithmetic cost, including eigenvalue computation, solution of linear systems etc \cite{Ng_book}, thanks to the fast Fourier transform algorithm \cite{Van_Loan}. For this reason and because of well-established spectral approximation results \cite{GLT-bookI,Ng_book}, in the last forty years, there has been a huge amount of work on preconditioning and multigrid using circulants or other matrix $*$-algebras associated to fast transforms.

In the following, we show that procedures inspired by nonstandard Pythagorean idea induces new lines of research, with practical and promising implications.

 For any square $n \times n$ matrix $ X = \left[ x_{ j, k } \right] $ and $ p \geq 1 $, consider the matrix norm defined as
\begin{equation}\label{lp-norm}
	\| X \|_{ l^{ p } } = \left( \sum_{ j, k = 1 }^{ n } \abs{ x_{ j, k } }^{ p } \right)^{ 1 / p }.
\end{equation}
When $p=2$, this $l^p$ norm coincides with the Schatten $2$-norm \cite{Bhatia}, or Frobenius norm, and it is induced by the positive Frobenius scalar product so that the space of $n\times n$ matrices with this scalar product is a Hilbert space. 
The Frobenius norm is very popular in numerical analysis for its favourable computational features. Here we consider the $l^p$ versions of the Frobenius norm, looking at its algebraic expression in the case of structures of Toeplitz type. For the sake of simplicity, we focus our attention on the one-level setting, even if the very same approaches can be easily generalized to multilevel block structures \cite{GLT-bookIV}. We notice that true $L^p$ version is that induced by the Schatten $p$-norms \cite{Bhatia}: however, it is computationally expensive to deal with them, and hence they are not treated here.

Given any complex matrix $X$ and by considering the Frobenius norm, we define
\[
C_{ n }^{ \left( 2 \right) }(X) = \arg \min_{ C_{ n } \text{ circulant} } \| X - C_{ n } \|_{ l^{2} }.
\]
The minimizer $C_{ n }^{ \left( 2 \right) }(X)$ exists and it is unique because circulant matrices form a finite-dimensional space (and hence closed and convex) and because of the unique Hilbert projection. Furthermore, for any complex matrix $X$ of size $n$, the corresponding Pythagorean relation
$\|X\|_{ l^{2} }^2=\| X - C_{ n }^{ \left( 2 \right) }(X) \|_{ l^{2} }^2+\|C_{ n }^{ \left( 2 \right) }(X) \|_{ l^{2} }^2$ holds.

Inspired by the reasoning above we consider a general $T_n(f)$ with Lebesgue integrable generating function $f$ and its circulant optimal approximation in $l^p$ norm for any $p\in [1,\infty)$ i.e. we consider $ C_{ n }^{ \left( p \right) }(f)$ as the best circulant approximation of $ T_{ n } \left( f \right) $ in the $ l^{ p } $ norm defined above, i.e.
\begin{equation}\label{minimizer-toe}
	C_{ n }^{ \left( p \right) }(f) = \arg \min_{ C_{ n } \text{ circulant} } \| T_{ n } \left( f \right) - C_{ n } \|_{ l^{ p } }.
\end{equation}
Now the Pythagorean equality $\|X\|_{ l^{2} }^2=\| X - C_{ n }^{ \left( 2 \right) }(f) \|_{ l^{2} }^2+\|C_{ n }^{ \left( 2 \right) }(f) \|_{ l^{2} }^2$ stands if and only if  $C_{ n }^{ \left( 2 \right) }(f)$ is the unique minimizer in (\ref{minimizer-toe}) with $p=2$.
Is there a generalized Pythagorean theorem for $p\neq 2$? And more importantly, is this approximation easily computable and competitive with the standard Frobenius choice with $p=2$?

\subsubsection{$ C_{ n }^{ \left( p \right) } $ as a preconditioner for Toeplitz systems}

In this section we investigate the effectiveness of $ C_{ n }^{ \left( p \right) }(f) $ as a preconditioner in the conjugate gradient method for linear Toeplitz systems. We start with a specific class of problems, designed to encode Toeplitz matrices with three levels of conditioning that are $O(1), O(n^2), O(n^4)$, while maintaining, for simplicity, a banded nature.

More precisely, let
\begin{equation} \label{def_f_1}
	f_{ \alpha, \beta, \gamma } \left( \theta \right) = \alpha + \beta \left( 2 - 2 \cos \theta \right) + \gamma \left( 2 - 2 \cos \theta \right)^{ 2 },
\end{equation}
with $\alpha, \beta, \gamma \ge 0$, $\beta \gamma \neq 0$, and let $ T_{ n } \left( f_{ \alpha, \beta, \gamma } \right) $ be the associated Toeplitz matrix. We observe that the minimum of $f_{ \alpha, \beta, \gamma }$ is exactly $\alpha$ while the maximum is $\alpha + 4\beta + 16\gamma$, due to the condition $\alpha, \beta, \gamma \ge 0$, to the even nature of the functions $\left( 2 - 2 \cos \theta \right)^{k}$, $k=0,1,2$, to the monotonic behavior in the interval $[0,\pi]$. Furthermore,  $\alpha <\alpha + 4\beta + 16\gamma$ because of the additional condition $\beta \gamma \neq 0$.
Under the inequality stated above, we know that  $ T_{ n } \left( f_{ \alpha, \beta, \gamma } \right) $ is positive definite with all the eigenvalues in the open interval $(\alpha,\alpha + 4\beta + 16\gamma)$ with the minimal eigenvalue of $ T_{ n } \left( f_{ \alpha, \beta, \gamma } \right) $ forming a monotonic strictly decreasing sequence converging to $\alpha$ and with the maximal eigenvalue of $ T_{ n } \left( f_{ \alpha, \beta, \gamma } \right) $ forming a monotonic strictly increasing sequence converging to $\alpha+ 4\beta + 16\gamma$.
Therefore, for $\alpha>0$ the Euclidean condition number converges to $\frac{\alpha+ 4\beta + 16\gamma}{\alpha}$ while for $\alpha=0$ and $\beta>0$ the Euclidean condition number divided by $n^2$ converges to a positive constant, since $f_{ \alpha, \beta, \gamma }$ has a unique zero at zero of order two. Finally, when $\alpha=\beta=0$ and $\gamma>0$, the Euclidean condition number divided by $n^4$ converges to a positive constant, since $f_{ \alpha, \beta, \gamma }$ has a unique zero at zero of order four; see \cite{extreme1,extreme2,GLT-bookI} and references therein.

Let $ C_{ n }^{ \left( p \right) }(f) $ be the best circulant approximation of $ T_{ n } \left( f \right) $ in the $ l^{ p } $ norm defined above, i.e.
\begin{equation*}
	C_{ n }^{ \left( p \right) }(f) = \arg \min_{ C_{ n } \text{ circulant} } \| T_{ n } \left( f \right) - C_{ n } \|_{ l^{ p } },
\end{equation*}
with $f=f_{ \alpha, \beta, \gamma }$.
Clearly, the latter minimization problem is equivalent to
\begin{equation} \label{minim_1}
	C_{ n }^{ \left( p \right) }(f) = \arg \min_{ C_{ n } \text{ circulant} } \| T_{ n } \left( f \right) - C_{ n } \|_{ l^{ p } }^{ p }.
\end{equation}
For simplicity, put $ \phi = \alpha + 2 \beta + 6 \gamma $ and $ \psi = - \beta - 4 \gamma $, so that $ f_{ \alpha, \beta, \gamma } \left( \theta \right) = f \left( \theta \right) = \phi + \psi \left( e^{ \hat i \theta } + e^{ - \hat i \theta } \right) + \gamma \left( e^{ 2 \hat i \theta } + e^{ - 2 \hat i \theta } \right) $, and denote as $ c = \left[ c_{ j }^{ \left( p \right) } \right]_{ j = 0 }^{ n - 1 } $ the first column of $ C_{ n }^{ \left( p \right) } $. \\
Some computations lead to an explicit formulation of the minimization problem \eqref{minim_1} as
\begin{equation*}
	\begin{cases}
		c_{ 0 }^{ \left( p \right) } = \phi \\
		c_{ 1 }^{ \left( p \right) } = \arg \min_{ c \in \R } \left( n - 1 \right) \abs{ \psi - c }^{ p } + \abs{ c }^{ p } \\
		c_{ 2 }^{ \left( p \right) } = \arg \min_{ c \in \R } \left( n - 2 \right) \abs{ \gamma - c }^{ p } + 2 \abs{ c }^{ p } \\
		c_{ n - 2 }^{ \left( p \right) } = c_{ 2 }^{ \left( p \right) } \\
		c_{ n - 1 }^{ \left( p \right) } = c_{ 1 }^{ \left( p \right) } \\
		c_{ j }^{ \left( p \right) } = 0 \text{ for } j = 3, \dots n - 3
	\end{cases},
\end{equation*}
with explicit solution
\begin{equation*}
	\begin{cases}
		c_{ 0 }^{ \left( p \right) } = \phi \\
		c_{ 1 }^{ \left( p \right) } = c_{ n - 1 }^{ \left( p \right) } =
		\begin{cases}
			\frac{ \psi \left( n - 1 \right)^{ 1 / \left( p - 1 \right) } }{ 1 + \left( n - 1 \right)^{ 1 / \left( p - 1 \right) } } \text{ for } p > 1 \\
			\psi \text{ for } p = 1
		\end{cases} \\
		c_{ 2 }^{ \left( p \right) } = c_{ n - 2 }^{ \left( p \right) } =
		\begin{cases}
			\frac{ \gamma }{ 1 + \left( \frac{ 2 }{ n - 2 } \right)^{ 1 / \left( p - 1 \right) } } \text{ for } p > 1 \\
			\gamma \text{ for } p = 1
		\end{cases} \\
		c_{ j }^{ \left( p \right) } = 0 \text{ for } j = 3, \dots n - 3
	\end{cases}.
\end{equation*}
Note that $ C_{ n }^{ \left( p \right) } $ is then symmetric.

Now we investigate the effectivness of $ C_{ n }^{ \left( p \right) }(f_{ \alpha, \beta, \gamma }) $ as a preconditioner in the conjugate gradient menthod for Toeplitz linear systems, whose coefficient matrices are $T_n(f)$ with $ f = f_{ \alpha, \beta, \gamma } $ as in \eqref{def_f_1}, for varying choices of the nonnegative parameters $ \alpha $, $ \beta $ and $ \gamma $. \\
First of all we notice that an explicit formula for the spectrum of $ C_{ n }^{ \left( p \right) } $ is available as
\begin{equation*}
	\Lambda \left( C_{ n }^{ \left( p \right) } \right) = \sqrt{ n } F_{ n }^{ * } c,
\end{equation*}
where $ F_{ n } = n^{ -1 } \left[ \exp \left( - \i \frac{ 2 \pi j k }{ n } \right) \right]_{ j, k = 0 }^{ n - 1 } $ is the normalized Fourier matrix and $ c $, as before, the first column of $ C_{ n }^{ \left( p \right) } $; hence, for $ j = 1, \dots n $,
\begin{eqnarray*}
	\lambda_{ j } \left( C_{ n }^{ \left( p \right) } \right) & = & \sum_{ k = 0 }^{ n - 1 } \exp \left( - \i \frac{ 2 \pi \left( j - 1 \right) k }{ n } \right) c_{ j } = \sum_{ k \in \left\lbrace 0, 1, 2, n - 2, n - 1 \right\rbrace } \exp \left( - \i \frac{ 2 \pi \left( j - 1 \right) k }{ n } \right) c_{ j } \\
																														& = & c_{ 0 } + 2 c_{ 1 } \cos \frac{ 2 \pi \left( j - 1 \right) }{ n } + 2 c_{ 2 } \cos \frac{ 4 \pi (j-1) }{ n },
\end{eqnarray*}
i.e.
\begin{equation}\label{spett_C}
\lambda_{ j } \left( C_{ n }^{ \left( p \right) }(f) \right) = f_{ \alpha, \beta, \gamma }(\theta_j^{ \left( n \right) })+
2\left((c_{ 1 }^{ \left( p \right) }-\psi)\cos\left(\theta_j^{ \left( n \right) }\right)\right) +
2\left((c_{ 2 }^{ \left( p \right) }-\gamma)\cos\left(2\theta_j^{ \left( n \right) }\right)\right).
\end{equation}
Since $c_{ 1 }^{ \left( p \right) }$ converges to $\psi$ as $n$ tends to infinity and $c_{ 2 }^{ \left( p \right) }$ converges to $\gamma$ as $n$ tends to infinity, the number of potential negative eigenvalues of $C_{ n }^{ \left( p \right) }(f)$ are $O(1)$ and they have infinitesimal modulus, as $n$ tends to infinity: in that case a $O(1)$ rank correction of $C_{ n }^{ \left( p \right) }(f)$ leads to a positive definite preconditioner using a Strang-type correction; see \cite{CN-Sirev} and references therein.

Since $ T_{ n } \left( f \right) $ is real symmetric positive definite and of finite bandwidth (or equivalently generated by a nonnegative even trigonometric polynomial), it may be regarded as a small rank plus small rank correction of a circulant matrix $ C_{ n }$, whose first column
${\bf c}_1$  is given by $ {\bf c}_1 = \left[ \phi, \psi, \gamma, 0 \dots, 0, \gamma, \psi \right]^{ T } $. \\
If we assume $ C_{ n }^{ \left( p \right) }(f) $ to be positive definite as well, that is imposing that the quantities in (\ref{spett_C}) are positive, then, by classical results on preconditioning \cite{CN-Sirev,Koro1} and up to a negligible amount of eigenvalues, it holds that $C_{ n }^{ \left( p \right), -1 } T_{ n }$ can be written as the identity plus a low norm correction term plus a small rank correction term. Hence, with $ C_{ n }^{ \left( p \right), -1 } $ denoting the inverse of $ C_{ n }^{ \left( p \right) } $, by a standard use of interlacing theorems as in
\cite{Bhatia,GLT-bookI}, we know that the sets
\begin{equation*}
	\left\{\lambda_{ j } \left( C_{ n }^{ \left( p \right), -1 }(f) T_{ n }(f) \right) \right\}
\end{equation*}
cluster to $1$, as $n$ tends to infinity, so ensuring fast convergence. However the dependency on $p$ can not be easily studied. In this respect, the numerical results give a clear indication that in some cases the lower is $p$ the better the clustering and the convergence speed;
see Table  \ref{tab:tab_02} and Table \ref{tab:tab_03} and the wide experimentation in the numerical section of \cite{beyond}.

\begin{table}[h]
	\caption{$ \left( \alpha, \beta, \gamma \right) = \left( 1, 2, 3 \right) $}
	\begin{tabular}{ll|llllllll}
		& $ p $ & $ 1 $ & $ 1.4 $ & $ 1.6 $ & $ 1.8 $ & $ 3 $  & $ 5 $  & $ 10 $ & n. p. \\
		$ n $    &       &       &         &         &         &        &        &        &                    \\
		$ 100 $  &       & $ 3 $ & $ 4 $   & $ 5 $   & $ 6 $   & $ 17 $ & $ 28 $ & $ 36 $ & $ 50 $             \\
		\hline
		$ 400 $  &       & $ 3 $ & $ 4 $   & $ 4 $   & $ 5 $   & $ 13 $ & $ 23 $ & $ 32 $ & $ 74 $             \\
		\hline
		$ 700 $  &       & $ 3 $ & $ 4 $   & $ 4 $   & $ 5 $   & $ 11 $ & $ 22 $ & $ 31 $ & $ 73 $             \\
		\hline
		\hline
		$ 1000 $ &       & $ 3 $ & $ 4 $   & $ 4 $   & $ 5 $   & $ 10 $ & $ 21 $ & $ 31 $ & $ 73 $
	\end{tabular}
	\caption*{Iteration results for the well-conditioned problem $ \left( \alpha, \beta, \gamma \right) = \left( 1, 2, 3 \right) $; here n. p. stands for non preconditioned.}
	\label{tab:tab_02}
\end{table}
\begin{table}[H]
	\caption{$ \left( \alpha, \beta, \gamma \right) = \left( 0, 2, 8 \right) $}
	\begin{tabular}{ll|llllllll}
		& $ p $ & $ 1 $ & $ 1.4 $ & $ 1.6 $ & $ 1.8 $ & $ 3 $  & $ 5 $  & $ 10 $ & n. p. \\
		$ n $    &       &       &         &         &         &        &        &        &                    \\
		$ 100 $  &       & $ \# $ & $ \# $   & $ 11 $   & $ 17 $   & $ 38 $ & $ 48 $ & $ 53 $ & $ 100 $             \\
		\hline
		$ 400 $  &       & $ \# $ & $ \# $   & $ 13 $   & $ 24 $   & $ 107 $ & $ 195 $ & $ 272 $ & $ \# $             \\
		\hline
		$ 700 $  &       & $ \# $ & $ \# $   & $ 13 $   & $ 29 $   & $ 168 $ & $ 344 $ & $ 498 $ & $ \# $             \\
		\hline
		$ 1000 $ &       & $ \# $ & $ \# $   & $ 15 $   & $ 32 $   & $ 223 $ & $ 485 $ & $ 729 $ & $ \# $
	\end{tabular}
	\caption*{Iteration results for the ill-conditioned problem $ \left( \alpha, \beta, \gamma \right) = \left( 0, 2, 8 \right) $; here n. p. stands for non-preconditioned and $ \# $ indicates the method failing due to matrix singular to machine precision.}
	\label{tab:tab_03}
\end{table}

\begin{remark}
	Heuristic observations lead to the following intuition: the most effective $ p $-norm circulant preconditioner (\ref{minim_1}) for
	$T_n= T_{ n } \left( f_{ \alpha, \beta, \gamma } \right)  $ coincides with $ C_{ n }^{ \left( \widetilde{ p } \right) }=C_{ n }^{ \left( \widetilde{ p } \right) }  \left( f_{ \alpha, \beta, \gamma } \right)$, with
	\begin{equation*}
		\widetilde{ p } = \min \left\lbrace p \geq 1 \text{ such that } \forall j = 1, \dots, n,\  \>\> \abs{ \Im \left( \lambda_{ j } \left( C_{ n }^{ \left( p \right), -1 } T_{ n } \right) \right) } < \epsilon \right\rbrace,
	\end{equation*}
	where $ \epsilon > 0 $ is chosen small. \\
	For instance, in the case summarized in  Table \ref{tab:tab_03}, we observe that PCG fails when preconditioned with $ C_{ n }^{ \left( p \right) } $, $ p = 1, 1.4 $ and achieves good performance with $ C_{ n }^{ \left( p \right) } $, $ p = 1.6 $. Indeed, as reported in \cite{beyond}, the imaginary parts spoil the PCG convergence for $ p = 1, 1.4 $ while the convergence is very good for $p=1.6$, with imaginary parts of the order of $ \sim 10^{ - 4 } $. \\
	The condition of a purely real positive spectrum of $ C_{ n }^{ \left( p \right), -1 } T_{ n } $ is implied by the positive definiteness of $ C_{ n }^{ \left( p \right) } $, which can be checked via the quantities in (\ref{spett_C}).
	
Furthermore, as reported in  Table \ref{tab:tab_02}, the fact that the case $p=1$ is optimal for some values of the parameters opens the door to consider a functional in (\ref{lp-norm}) with $p\in (0,1)$. Even if the considered functional is not a norm (the set of the matrices $X$ such that $\|X\|_{ l^{ p } }\le 1$ is not convex), the minimization process can be considered and it is not excluded that it can lead to computationally effective and better results, mimicking what it is known in imaging \cite{imag1,imag2}.
\end{remark}


\section{Conclusion}\label{sec:Conclusion}

In the present work, we have indicated possible generalization of the notions of angles and orthogonality in generic Banach spaces, e.g. $L^p$, $p\in [1,\infty)$. When considering the intrinsic mode functions decomposition of non-stationary signals done in signal processing, we find a clear example of conservation of energy, which can be read as the Pythagorean theorem in $L^1$. We have provided an initial theoretical analysis and, inspired by these new concepts, we have proposed new preconditioning proposals in the context of numerical methods for the fast solution of structured large linear systems. A few numerical examples have been presented and critically discussed, together with various remarks on possible further use of the given framework. We mention that the generalization of the notions of angles and orthogonality presented in this paper can be further extended to the case of normed linear spaces. We plan to do so in future work.

\section*{Acknowledgment}
AC, SSC, and GT are members of the Gruppo Nazionale Calcolo Scientifico-Istituto Nazionale di Alta Matematica (GNCS-INdAM).

The research of AC was partially supported through the GNCS-INdAM Project, CUP E53C23001670001. AC was supported by the Italian Ministry of the University and Research and the European Union through the ``Next Generation EU'', Mission 4, Component 1, under the PRIN PNRR 2022 grant number CUP E53D23018040001 ERC field PE1 project P2022XME5P titled ``Circular Economy from the Mathematics for Signal Processing prospective''.

Furthermore, SSC has been supported by the Theory, Economics and Systems Laboratory (TESLAB) of the Department of Computer Science of the Athens University of Economics and Business and by the PRIN-PNRR project MATHPROCULT (``MATHematical tools for predictive maintenance and PROtection of CULTural heritage'', Code P20228HZWR, CUP J53D23003780006). 

The research of HZ has been partially supported by NSF DMS-2307465 and DMS-2510829, and the Elaine M. Hubbard fellowship.

\bibliographystyle{plain}
\bibliography{Pythagoras}

@article{cicone2024new,
  title={{New theoretical insights in the decomposition and time-frequency representation of nonstationary signals: the IMFogram algorithm}},
  author={Cicone, Antonio and Li, Wing Suet and Zhou, Haomin},
  journal={Applied and Computational Harmonic Analysis},
  volume={71},
  pages={101634},
  year={2024},
  publisher={Elsevier}
}

@article{lumer1961semi,
  title={{Semi-inner-product spaces}},
  author={Lumer, G{\"u}nter},
  journal={Transactions of the American Mathematical Society},
  volume={100},
  number={1},
  pages={29--43},
  year={1961},
  publisher={JSTOR}
}

@article{Giles1967,
  author    = {Giles, J. R.},
  title     = {Classes of semi-inner-product spaces},
  journal   = {Transactions of the American Mathematical Society},
  year      = {1967},
  volume    = {129},
  number    = {3},
  pages     = {436--446},
  doi       = {10.1090/S0002-9947-1967-0217574-1}
}

@book{Dragomir2004,
  author    = {Dragomir, Sever Silvestru},
  title     = {Semi-Inner Products and Applications},
  publisher = {Nova Science Publishers, Inc.},
  address   = {Hauppauge, NY},
  year      = {2004},
  isbn      = {1-59033-947-9}
}

@article{cicone2022multivariate,
  title={{Multivariate fast iterative filtering for the decomposition of nonstationary signals}},
  author={Cicone, Antonio and Pellegrino, Enza},
  journal={IEEE Transactions on Signal Processing},
  volume={70},
  pages={1521--1531},
  year={2022},
  publisher={IEEE}
}

@article{gustafson1968angle,
  title={{The angle of an operator and positive operator products}},
  author={Gustafson, Karl},
  journal={Bulletin of the American Mathematical Society},
  volume={74},
  issue={3},
  pages={488-492},
  year={1968}
}

@book{axler1997numerical,
  title={Numerical range: The field of values of linear operators and matrices},
  author={Gustafson,  Karl and Rao, Duggirala K.M. },
  year={1997},
  publisher={Springer New York}
}

@article{lin2009iterative,
  title={{Iterative filtering as an alternative algorithm for empirical mode decomposition}},
  author={Lin, Luan and Wang, Yang and Zhou, Haomin},
  journal={Advances in Adaptive Data Analysis},
  volume={1},
  number={04},
  pages={543--560},
  year={2009},
  publisher={World Scientific}
}

@article{huang1998empirical,
  title={{The empirical mode decomposition and the Hilbert spectrum for nonlinear and non-stationary time series analysis}},
  author={Huang, Norden E and Shen, Zheng and Long, Steven R and Wu, Manli C and Shih, Hsing H and Zheng, Quanan and Yen, Nai-Chyuan and Tung, Chi Chao and Liu, Henry H},
  journal={Proceedings of the Royal Society of London. Series A: mathematical, physical and engineering sciences},
  volume={454},
  number={1971},
  pages={903--995},
  year={1998},
  publisher={The Royal Society}
}

@article{cicone2021numerical,
  title={{Numerical analysis for iterative filtering with new efficient implementations based on FFT}},
  author={Cicone, Antonio and Zhou, Haomin},
  journal={Numerische Mathematik},
  volume={147},
  number={1},
  pages={1--28},
  year={2021},
  publisher={Springer}
}

@article{GLT-bookIV,
  title={{Block generalized locally Toeplitz sequences: theory and applications in the multidimensional case}},
  author={Barbarino, Giovanni and Garoni, Carlo and Stefano Serra-Capizzano},
  journal={Electronic Transactions on Numerical Analysis},
  volume={53},
  number={1},
  pages={113--216},
  year={2020},
  publisher={Kent State University, Johann Radon Institute}
}

@book{Bhatia,
  title={Matrix analysis},
  author={Bhatia, Rajendra},
  year={1997},
  publisher={Springer}
}

@book{GLT-bookI,
  title={Generalized locally Toeplitz sequences: theory and applications. Vol. I.},
  author={Garoni, Carlo and Stefano Serra-Capizzano},
  year={2017},
  publisher={Springer}
}

@article{sparse_appr1,
  title={{Approximate sparsity patterns for the inverse of a matrix and preconditioning}},
  author={Huckle, Thomas},
  journal={Applied Numerical Mathematics},
  volume={30},
  number={2},
  pages={291-303},
  year={1999},
  publisher={Elsevier}
}

@article{sparse_appr2,
  title={{Nesting approximate inverses for improved preconditioning and algebraic multigrid smoothing}},
  author={Janna, Carlo and Franceschini, Andrea},
  journal={SIAM Journal on Matrix Analysis and Applications},
  volume={46},
  number={1},
  pages={393-415},
  year={2025},
  publisher={Elsevier}
}

@book{Van_Loan,
  title={Computational Frameworks for the Fast Fourier Transform},
  author={Van Loan, Charles},
  year={1992},
  publisher={Society for Industrial and Applied Mathematics}
}

@book{Ng_book,
  title={Iterative methods for Toeplitz systems},
  author={Ng, Michael K.},
  year={2004},
  publisher={Oxford University Press}
}

@article{extreme1,
  title={{On the extreme spectral properties of Toeplitz matrices generated by L1 functions with several minima/maxima}},
  author={Serra-Capizzano, Stefano},
  journal={BIT Numerical Mathematics},
  volume={36},
  number={1},
  pages={135-142},
  year={1996},
  publisher={Springer}
}

@article{extreme2,
  title={{On the condition numbers of large semi-definite Toeplitz matrices}},
  author={Böttcher, Albrecht and Grudsky, Sergei M.},
  journal={Linear Algebra and Its Applications},
  volume={279},
  number={1-3},
  pages={285-301},
  year={1998},
  publisher={Elsevier}
}

@article{CN-Sirev,
  title={{Conjugate gradient methods for Toeplitz systems}},
  author={Chan, Raymond H. and Ng, Michael K.},
  journal={SIAM Review},
  volume={38},
  number={3},
  pages={427-482},
  year={1996},
  publisher={Elsevier}
}

@article{Koro1,
  title={{A Korovkin-type theory for finite Toeplitz operators via matrix algebras}},
  author={Serra-Capizzano, Stefano},
  journal={Numerische Mathematik},
  volume={82},
  number={1},
  pages={117-142},
  year={1999},
  publisher={Springer}
}

@article{imag1,
  title={{A generalized Krylov subspace method for $\ell_p$-$\ell_q$ minimization}},
  author={Lanza, Alessandro and Morigi, Serena and Reichel, Lothar and Sgallari, Fiorella},
  journal={SIAM Journal on Scientific Computing},
  volume={37},
  number={5},
  pages={S30-S50},
  year={2015},
  publisher={Elsevier}
}

@article{imag2,
  title={{Modulus-based iterative methods for constrained $\ell_p$-$\ell_q$ minimization}},
  author={Buccini, Alessandro and Pasha, Mirjeta and Reichel, Lothar},
  journal={Inverse Problems},
  volume={36},
  number={8},
  pages={084001},
  year={2020},
  publisher={IOP Publishing}
}

@article{beyond,
  title={{Minimization processes, Pythagorean theorem, and Toeplitz preconditioning}},
  author={Cicone, Antonio and Serra-Capizzano, Stefano and Tento, Giacomo and Zhou Haomin},
  journal={Preprint},
  year={2026}
}

@article{Longo2025AdaptiveSL,
  title={Adaptive scattered light noise subtraction in GW detectors},
  author={Alessandro Longo and Gabriele Demasi and Francesco Di Renzo and Stefano Bianchi and Guillermo Valdes and Nicolas Arnaud and Gianluca Inguglia and Antonio Cicone and Massimo Lenti and Francesco Bucci and Giulio Settanta and Matteo Montani and Roberto Cavassi},
  journal={Classical and Quantum Gravity},
  year={2025},
  url={https://api.semanticscholar.org/CorpusID:281646617}
}

@article{spogli2025investigating,
  title={Investigating the drivers of long-term trends in the upper atmosphere over Rome across four decades},
  author={Spogli, Luca and Sabbagh, Dario and Perrone, Loredana and Scotto, Carlo and Cesaroni, Claudio},
  journal={Journal of Space Weather and Space Climate},
  volume={15},
  pages={8},
  year={2025},
  publisher={EDP Sciences}
}

\end{document}